\documentclass[12pt]{article}
\usepackage{amsmath, amsthm, amsfonts, amscd, amssymb, graphicx}
\usepackage[letterpaper, total={6in, 8in}]{geometry}
\usepackage{enumerate}
\usepackage{setspace}
\usepackage{epstopdf}
\newtheorem{theorem}{\bf Theorem}[section]

\newtheorem{remark}{\bf Remark}[section]
\newcommand{\eps}{\varepsilon}
\newcommand{\cg}{\mathcal{C}_1^g}
\newcommand{\tr}{\mathrm{tr}\,}
\newcommand{\bn}{\mathbf{n}}

\title{Dimension reduction for the Landau-de Gennes model in planar nematic thin films}

\author{Dmitry Golovaty\\
\small Department of Mathematics\\
\small The University of Akron\\
\small Akron, OH 4325-4002, USA\\
\small \texttt{dmitry@uakron.edu}\\
\and
Jos\'e Alberto Montero\\
\small Departamento de Matem\'aticas\\
\small Facultad de Matem\'aticas\\
\small Pontificia Universidad Cat\'olica de Chile\\
\small Vicu\~na Mackenna 4860\\
\small San Joaqu\'in, Santiago, Chile\\
\small \texttt{amontero@mat.puc.cl}
\and 
Peter Sternberg\\
\small Department of Mathematics\\
\small Indiana University\\
\small Bloomington, IN 47405\\
\small \texttt{sternber@indiana.edu}
}
\date{}

\begin{document}
\maketitle

\begin{abstract}
We use the method of $\Gamma$-convergence to study the behavior of the Landau-de Gennes model for a nematic liquid crystalline film in the limit of vanishing thickness. In this asymptotic regime, surface energy plays a greater role and we take particular care in understanding its influence on the structure of the minimizers of the derived two-dimensional energy. We assume general weak anchoring conditions on the top and the bottom surfaces of the film and the strong Dirichlet boundary conditions on the lateral boundary of the film. The constants in the weak anchoring conditions are chosen so as to enforce that a surface-energy-minimizing nematic $Q$-tensor has the normal to the film as one of its eigenvectors. We establish a general convergence result and then discuss the limiting problem in several parameter regimes.
\end{abstract}

\section{Introduction}
In this paper we study equilibrium configurations of thin nematic liquid crystalline films. Nematic materials are typically composed of rod- or disk-like molecules and can flow like fluids, yet they retain a degree of molecular orientational order similar to crystalline solids. There are several mathematical frameworks to study the nematics, leading to different, but related variational models. The focus of the present work is on rigorous dimensional reduction of the Landau-de Gennes $Q$-tensor model from three to two dimensions.

We begin by briefly reviewing the basic theory of nematics. The local orientations of molecules in {\em uniaxial nematics} are described by a director---a unit vector in a direction preferred by the molecules at a given point. The director field underlies the Oseen-Frank theory \cite{virga} based on an energy penalizing for spatial variations of the director.  This theory incorporates various elastic modes (splay, bend, twist) and interactions with electromagnetic fields and has generally been very successful in predicting equilibrium nematic configurations. However the Oseen-Frank approach is limited in that it prohibits certain types of topological defects, e.g., disclinations, as the constraint that the director must have a unit length is too rigid. A possible remedy was proposed by Ericksen \cite{ericksen} who introduced a scalar parameter intended to describe the quality of local molecular orientational order.

Despite the fact that Ericksen's theory is capable of handling line defects, it still assumes that a preferred direction is specified by the director, excluding a possibility that the nematic can be biaxial. Here a biaxial state differs from a uniaxial state in that it has no rotational symmetry; instead it possesses reflection symmetries with respect to each of three orthogonal axes (only two of which need to be specified). Biaxial configurations are conjectured to exist, e.g., at the core of a nematic defect \cite{kralj2001universal}. Further, certain nematic configurations cannot even be orientable, that is, they cannot be described by a continuous director field \cite{Ball_Zarn}. These deficiencies can be circumvented within the Landau-de Gennes theory in three dimensions that we will discuss in subsequent sections (see also \cite{Ball_Zarn}, \cite{apala_zarnescu_01}, and \cite{Mottram_Newton}). Briefly, the Landau-de Gennes theory is based on the $Q$-tensor order parameter field that is related to the second moment of local orientational probability distribution. The relevant variational model involves minimization of an energy functional consisting of elastic, bulk and surface contributions. Recently there has been considerable activity 
on modeling with surface energy terms using both $Q$-tensor and director theories, including for example, \cite{apala_zarnescu_01,canevari13,Lamy14,GM,segatti14,ball2010nematic}.

We are interested in proper reduction of the Landau-de Gennes model to two dimensions in the thin film limit. In this asymptotic regime, surface energy plays a greater role and we take particular care in understanding its influence on the structure of the minimizers of the derived two-dimensional energy. To achieve this goal we employ the tool of $\Gamma$-convergence that has proved successful in tackling problems of dimension reduction in other settings, such as elasticity \cite{anzellotti1994dimension} and Ginzburg-Landau theory \cite{contreras2010gamma}.  We work in the domain $\Omega\times(0,h)$ where $0<h\ll1$ and $\Omega\subset\mathbb R^2$ is bounded and Lipschitz. In a subsequent publication we extend this analysis to the case of a small neighborhood of an arbitrary smooth surface, either with or without boundary, as a rigorous analog of the dimension reduction procedure in \cite{Napoli_Vergori} (see also \cite{virga_talk}).

In Section \ref{s:surf} we introduce and analyze the general expression for the surface energy and then, in Section \ref{s:model} combine it with the bulk and elastic terms to form the full non-dimensionalized three dimensional energy functional. In Section \ref{s:conv} we derive the expression for the limiting energy $F_0$ and in Section \ref{s:reg} we analyze minimizers of $F_0$ in different parameter regimes. In all regimes that we consider, it is crucial that the space of competing $Q$-tensors is constrained to accommodate the requirement that each tensor has a normal to the surface of the film as an eigenvector. This condition is forced in the limit by the surface energy and justifies the reduced $Q$-tensor ansatz imposed in \cite{PhysRevLett.59.2582,bauman_phillips_park} in relation to experiments in \cite{PhysRevE.66.030701}.

\section{The $Q$-tensor}
In the three-dimensional setting, one describes the nematic liquid crystal by a $2$-tensor $Q$ which takes the form of a $3\times  3$ symmetric, traceless matrix. Here $Q(x)$ models the second moment of the orientational distribution of the rod-like molecules near $x$. The tensor $Q$  has three real eigenvalues satisfying $\lambda_1+\lambda_2+\lambda_3=0$ and a mutually orthonormal eigenframe $\left\{\mathbf{l},\mathbf{m},\mathbf{n}\right\}$.

Suppose that $\lambda_1=\lambda_2=-\lambda_3/2.$ Then the liquid crystal is in a {\em uniaxial nematic} state and \begin{equation}Q=-\frac{\lambda_3}{2}\mathbf{l}\otimes\mathbf{l}-\frac{\lambda_3}{2}\mathbf{m}\otimes\mathbf{m}+
\lambda_3\mathbf{n}\otimes\mathbf{n}=S\left(\mathbf{n}\otimes\mathbf{n}-\frac{1}{3}\mathbf{I}\right),\label{uniaxial}
 \end{equation}
 where $S:=\frac{3\lambda_3}{2}$ is the uniaxial nematic order parameter and $\mathbf{n}\in\mathbb{S}^2$ is the nematic director. If there are no repeated eigenvalues, the liquid crystal is in a {\em biaxial nematic} state and
\begin{multline}
Q=\lambda_1\mathbf{l}\otimes\mathbf{l}+\lambda_3\mathbf{n}\otimes\mathbf{n}-\left(\lambda_1+\lambda_3\right)\left(\mathbf{I}-\mathbf{l}\otimes\mathbf{l}-\mathbf{n}\otimes\mathbf{n}\right)\\=S_1\left(\mathbf{l}\otimes\mathbf{l}-\frac{1}{3}\mathbf{I}\right)+S_2\left(\mathbf{n}\otimes\mathbf{n}-\frac{1}{3}\mathbf{I}\right),
\label{biaxial}\end{multline}
where $S_1:=2\lambda_1+\lambda_3$ and $S_2=\lambda_1+2\lambda_3$ are biaxial order parameters.
Note that uniaxiality can also be described in terms of $S_1$ and $S_2$, that is one of the following three cases occurs: $S_1=0$ but $S_2\not=0$, $S_2=0$ but $S_1\not =0$ or $S_1=S_2\not=0.$ When $S_1=S_2=0$ so that ${\bf Q}=0$ the nematic liquid crystal is said to be in an isotropic state associated, for instance, with a high
temperature regime.

From the modeling perspective it turns out that the eigenvalues of $Q$ must satisfy the constraints \cite{ball2010nematic,sonnet2012dissipative}:
\begin{equation}
\label{eq:bnds}
\lambda_i\in[-1/3,2/3],\ \mathrm{for}\ i=1,2,3.
\end{equation}

\section{Surface energy}
\label{s:surf}
In this section we discuss the behavior of the nematic on the boundary of the sample. Here two alternatives are possible. First, the Dirichlet boundary conditions on $Q$ are referred to as strong anchoring conditions in the physics literature: they impose specific preferred orientations on nematic molecules on surfaces bounding the liquid crystal. In the sequel we impose these conditions on the lateral part of the cylindrical sample $\partial\Omega\times(0,h)$. An alternative is to specify the surface energy on the boundary of the sample; then orientations of the molecules on the boundary are determined as a part of the minimization procedure. We adopt this approach, referred to as {\em weak anchoring}, on the top and the bottom surfaces of the film.

We seek a general expression for surface energy that has a family of surface-energy-minimizing tensors with the normal to the surface of the liquid crystal as their eigenvector. The requirement that the normal to the film is also an eigenvector of the $Q$ tensor is motivated by the desire to model both homeotropic and parallel anchoring---corresponding to the nematic molecules oriented perpendicular and parallel to the surface of the film, respectively \cite{virga}. In the former case the uniaxial nematic tensor is prescribed on the boundary (up to the multiplicative order parameter) with the director being perpendicular to the surface of the film. In the latter case, the director orientation is perpendicular to the normal to the film but otherwise may be arbitrary.

Consider the "bare" surface energy (Eq. 7 in \cite{OH}, see also \cite{sen1987landau} and \cite{sluckin1986fluid})
\begin{equation}
\label{bare}
f_s(Q,\nu):=c_1(Q\nu\cdot\nu)+c_2Q\cdot Q+c_3(Q\nu\cdot\nu)^2+c_4{|Q\nu|}^2,
\end{equation}
where $c_i,\ i=1,\ldots,4$ are constants, $A\cdot B=\tr\left(B^TA\right)$ for any two $n\times n$ matrices $A$ and $B$, and $\nu\in\mathbb{S}^2$ is a normal to the surface of the liquid crystal. This expression can be supplemented, in principle, by the surface Landau-de Gennes-type expression (Eq. 75 in \cite{Mottram_Newton}) in order to control eigenvalues and to relax constraints on the constants $c_i$ in \eqref{bare} that will be imposed below. This would amount essentially to augmenting \eqref{bare} with an expression of the form \eqref{eq:LdG} below.

Fix $\nu\in\mathbb{S}^1$ and let
\begin{equation}
\label{eq:cala}
\mathcal A:=\left\{Q\in M^{3\times3}_{sym}:\mathrm{tr}\,{Q}=0\right\}. 
\end{equation}
We now explore different parameter regimes associated with minimization of the surface energy over the set $\mathcal A$. Some comments along these lines can be found in \cite{fournier2005modeling}.
\begin{theorem}
\label{thm:surf}
The minimum of $f_s(Q,\nu)$ over $Q\in\mathcal{A}$ is achieved in the following five cases as characterized below in terms of the parameters $c_1,\ldots,c_4$:
\begin{enumerate}[(i)]
\item If $\min\{c_2,2c_2+c_4,3c_2+2c_3+2c_4\}>0$, then the minimum of $f_s(Q,\nu)$ is achieved at any $Q$ that is uniaxial and homeotropic with the eigenvalue $-\frac{c_1}{3c_2+2c_3+2c_4}$ corresponding to the eigenvector $\nu$.
\item If $\min\{c_2,2c_2+c_4\}>0$ and $3c_2+2c_3+2c_4=c_1=0$, then the minimum of $f_s(Q,\nu)$ is achieved at any $Q$ that is uniaxial and homeotropic, i.e., with the nematic director parallel to $\nu$.
\item If $\min\{c_2,2c_3-c_2\}>0$ and $2c_2+c_4=0,$ then the minimum of $f_s(Q,\nu)$ is achieved at any $Q$ that satisfies $Q\nu\cdot\nu=\frac{c_1}{c_2-2c_3}$ and has one eigenvector orthogonal to $\nu$ with eigenvalue $\frac{c_1}{4c_3-2c_2}$.
\item If $\min\{2c_2+c_4,c_3+c_4\}>0$ and $c_2=0,$ then the minimum of $f_s(Q,\nu)$ is achieved at any uniaxial or biaxial $Q$ of the form
\begin{equation}
\label{eq:mu}
Q=\mu\mathbf{m}\otimes\mathbf{m}+\left(\frac{c_1}{2(c_3+c_4)}-\mu\right)\mathbf{n}\otimes\mathbf{n}-\frac{c_1}{2(c_3+c_4)}\nu\otimes\nu,
\end{equation}
where $\left\{\mathbf{m}, \mathbf{n}\right\}$ is an arbitrary orthonormal frame in the plane tangent to the surface of the liquid crystal, and $\mu\in\mathbb{R}$ is arbitrary. 
\item If $c_3>0$ and $c_2=c_4=0,$ then the minimum of $f_s(Q,\nu)$ is achieved and any minimizing $Q$ must satisfy $Q\nu\cdot\nu=-\frac{c_1}{2c_3}$.
\end{enumerate}
In all other cases, $\inf_{Q\in\mathcal{A}}f_s(Q,\nu)=-\infty$.

\end{theorem}
\begin{proof}
Without loss of generality we may assume that $\nu=(0,0,1):=\hat z$. In order to find $Q\in\mathcal A$ that minimizes $f_s$, observe that
\[Q\cdot Q=2{|Q\nu|}^2-{(Q\nu\cdot\nu)}^2+Q_2\cdot Q_2,\]
where $Q_2\in M^{2\times2}_{sym}$ is a nonzero square block of $\left(\mathbf{I}-\nu\otimes\nu\right)Q\left(\mathbf{I}-\nu\otimes\nu\right).$ Then the expression for $f_s$ can be written as
\begin{equation}
\label{bare_m}
f_s(Q,\nu)=c_1(x\cdot\nu)+c_2Q_2\cdot Q_2+\left(c_3-c_2\right)(x\cdot\nu)^2+\left(2c_2+c_4\right){|x|}^2,
\end{equation}
where $x:=Q\nu\in\mathbb{R}^3$. The traceless condition for $Q$ can be reformulated in terms of $Q_2$ and $x$ as
\begin{equation}
\label{trcond}
\mathrm{tr}\,Q_2+x\cdot\nu=0\,.
\end{equation}
Thus, we are looking for the minimum of \eqref{bare_m} among all $Q_2\in M^{2\times2}_{sym}$ and $x\in\mathbb{R}^3$ that satisfy the condition \eqref{trcond}. 

Suppose first that $c_2<0$. Then, setting $x=0$ and observing that
\[Q_2=
\left(
\begin{array}{cc}
 \alpha & 0  \\
 0 & -\alpha  \\
\end{array}
\right)
\]
satisfies the constraint \eqref{trcond} for any $\alpha\in\mathbb{R},$ we conclude that $\inf_{x,Q_2}f_s=-\infty.$ We leave the verification of other parameter regimes which result in  $\inf_{x,Q_2}f_s=-\infty$ to the reader and instead concentrate on the five cases laid out in the statement of the theorem. 

To this end, let $c_2>0$. Minimizing $f_s$ with respect to $Q_2$ subject to the constraint \eqref{trcond} is equivalent to minimizing $Q_2\cdot Q_2$ subject to \eqref{trcond}. We have
\begin{equation}
\label{eq:1}
2Q_2+\Lambda\mathbf{I}_2=0,
\end{equation}
where $\Lambda$ is the Lagrange multiplier. Combining \eqref{trcond} and \eqref{eq:1}, we find that $\Lambda=(x\cdot\nu)$ and 
\begin{equation}
\label{eq:Q2M}
Q_2=-\frac{1}{2}(x\cdot\nu)\mathbf{I}_2.
\end{equation}
Substituting this expression back into \eqref{bare_m}, we have
\[\inf_{\mathcal{A}} f_s=\inf_{\mathbb{R}^3}\tilde{f}_s,\]
where
\begin{equation}
\label{eq:exp1}
\tilde{f}_s(x,\nu):=c_1(x\cdot\nu)+\left(c_3-\frac{c_2}{2}\right)(x\cdot\nu)^2+\left(2c_2+c_4\right){|x|}^2,
\end{equation}
or, equivalently,
\begin{equation}
\label{eq:exp2}
\tilde{f}_s(x,\nu):=c_1(x\cdot\nu)+\left(\frac{3}{2}c_2+c_3+c_4\right)(x\cdot\nu)^2+\left(2c_2+c_4\right){\left|\left({\mathbf I}-\nu\otimes\nu\right)x\right|}^2,
\end{equation}
Observe that if $\min\left\{3c_2+2c_3+2c_4,2c_2+c_4\right\}<0$, then $\inf_{x}\tilde f_s=-\infty$. Suppose now that $\min\left\{3c_2+2c_3+2c_4,2c_2+c_4\right\}>0$ so that there is a critical point $x$ of $\tilde f_s$ that satisfies
\[
\frac{\partial \tilde{f}_s}{\partial x}=c_1\nu+\left(2c_3-c_2\right)(x\cdot\nu)\nu+2\left(2c_2+c_4\right)x=0.
\]
It follows that $x$ is parallel to $\nu$ and
\begin{equation}
\label{eq:3}
x=-\frac{c_1}{3c_2+2c_3+2c_4}\nu
\end{equation}
is a minimum. Consequently,
\begin{equation}
\label{eq:5}
Q\nu=\lambda\nu,
\end{equation}
where $\lambda=-\frac{c_1}{3c_2+2c_3+2c_4}$. Combining \eqref{eq:5} with the expression for $Q_2$, we find the single minimum
\begin{equation}
\label{eq:4}
Q=\frac{3\lambda}{2}\left(\nu\otimes\nu-\frac{1}{3}\mathbf{I}\right),
\end{equation}
of the surface energy corresponding to a fixed uniaxial nematic state with the order parameter $S=\frac{3\lambda}{2}$ and the nematic director $\nu$. This is the case of so-called homeotropic (perpendicular) anchoring. This establishes (i).

Proceeding with the proof of (ii), if $\min\{c_2,2c_2+c_4\}>0$ and $3c_2+2c_3+2c_4=c_1=0$ then the expression \eqref{eq:exp2} for $\tilde f_s$ reduces to its last term. Hence minimization simply requires $x$ to be parallel to $\nu$ meaning that $\nu$ is an eigenvector of $Q$. Further, $Q$ is uniaxial since the corresponding eigenvalue is $x\cdot\nu$ and $Q_2$ is given by \eqref{eq:Q2M}.

Next, if $\min\{c_2,2c_3-c_2\}>0$ and $2c_2+c_4=0,$ then \eqref{eq:exp1} reveals that $\tilde f_s$ is a quadratic function of $x\cdot\nu$ that is minimized when $x\cdot\nu=\frac{c}{c_2-2c_3}=:\sigma$. Combining this observation with \eqref{eq:Q2M} we obtain that the minimizing $Q$ must be of the form
\[
Q=\left(
\begin{array}{ccc}
 -\sigma/2 & 0  & q_{13}  \\
 0 & -\sigma/2  & q_{23}  \\
 q_{13} & q_{23}  & \sigma  
\end{array}
\right).
\]
From this we readily see that $\left(-q_{23},q_{13},0\right)$ is an eigenvector of $Q$. This establishes (iii). Note that the minimizing set of $f_s$ consists of a family of biaxial tensors that contains a single uniaxial representative corresponding to a homeotropic boundary condition. Here the biaxial tensors {\it do not} have a normal to the surface of the film as an eigenvector.

Now we pursue the regime $c_2=0$. In this case, from \eqref{bare_m} we see that $f_s$ is independent of $Q_2$. Thus the minimizing matrix $Q_2$ is arbitrary as long as it satisfies the trace constraint \eqref{trcond}. If, in addition, $\min\{2c_2+c_4,c_3+c_4\}>0$ the energy is still minimized by $x$ given by \eqref{eq:3} with $c_2=0$, i.e.,
\[Q\nu=-\frac{c_1}{2(c_3+c_4)}\nu.\]
It also follows that if $\left\{\mathbf{m}, \mathbf{n}\right\}$ is an arbitrary orthonormal frame in the plane tangent to the surface of the liquid crystal, then any
\[
Q=\mu\mathbf{m}\otimes\mathbf{m}+\left(\frac{c_1}{2(c_3+c_4)}-\mu\right)\mathbf{n}\otimes\mathbf{n}-\frac{c_1}{2(c_3+c_4)}\nu\otimes\nu
\]
minimizes the surface energy. This verifies (iv).

Finally in case (v) the expression \eqref{bare_m} reduces to a quadratic function of $x\cdot\nu$ that is minimized at $x\cdot\nu=-c_1/2c_3$. Observe that in this case the minimizing $Q_2$ is again arbitrary up to the trace constraint \eqref{trcond}.
\end{proof}

\begin{remark}
Note that the eigenvalues determined in the cases (i), (iii), and (iv) must respect the bounds \eqref{eq:bnds} on eigenvalues of $Q$ thereby imposing additional restrictions on the parameters $c_1,\ldots,c_4$.
\end{remark}

Having explored all possible parameter regimes we now focus on case (iv) where $\min\{2c_2+c_4,c_3+c_4\}>0$ and $c_2=0.$ Here the degeneracy of the set of tensors minimizing the surface energy $f_s$ provides sufficient freedom for nontrivial reduction to two-dimensional limits that we will carry out in the next section. An alternative approach would be to extend the surface energy by including quartic terms \cite{fournier2005modeling} and even surface derivative terms \cite{Longa}. 

By rearranging the terms in \eqref{bare}, the surface energy has the form
\begin{multline}
\label{fs}
f_s(Q,\nu)=c_1(Q\nu\cdot\nu)+c_3(Q\nu\cdot\nu)^2+c_4{|Q\nu|}^2 \\ =\left(c_3+c_4\right)\left[(Q\nu\cdot\nu)+\frac{c_1}{2(c_3+c_4)}\right]^2+c_4{\left|\left(\mathbf{I}-\nu\otimes\nu\right)Q\nu\right|}^2 \\ =\alpha\left[(Q\nu\cdot\nu)-\beta\right]^2+\gamma{\left|\left(\mathbf{I}-\nu\otimes\nu\right)Q\nu\right|}^2,
\end{multline}
up to an additive constant. Here $\alpha=c_3+c_4>0$, $\beta=-\frac{c_1}{2(c_3+c_4)}$, and $\gamma=c_4>0.$ This form of the surface energy explicitly demonstrates that the minimizer has $\nu$ as one eigenvector with corresponding eigenvalue equal to $\beta$. 

\section{Landau-de Gennes Energy Functional.\\ Non-dimensionalization}
\label{s:model}
Suppose that the bulk elastic energy density of a nematic liquid crystal is given by
\begin{equation}
\label{elastic}
f_e(\nabla Q):=\frac{L_1}{2}{|\nabla Q|}^2+\frac{L_2}{2}Q_{ij,j}Q_{ik,k}+\frac{L_3}{2}Q_{ik,j}Q_{ij,k}+\frac{L_4}{2}Q_{lk}Q_{ij,k}Q_{ij,l}
\end{equation}
and that the bulk Landau-de Gennes energy density is
\begin{equation}
\label{eq:LdG}
f_{LdG}(Q):=a\,\mathrm{tr}\left(Q^2\right)+\frac{2b}{3}\mathrm{tr}\left(Q^3\right)+\frac{c}{2}\left(\mathrm{tr}\left(Q^2\right)\right)^2,
\end{equation}
cf. \cite{Mottram_Newton}. Here the coefficient $a$ is temperature-dependent and in particular is negative for sufficiently low temperatures, and $c>0$.
One readily checks that the form \eqref{eq:LdG} of this potential implies that in fact $f_{LdG}$ depends only on the eigenvalues of $Q$, and due to the trace-free condition, therefore
 depends only on two eigenvalues. Equivalently, one can view $f_{LdG}$ as a function of the two degrees of orientation $S_1$ and $S_2$ appearing in \eqref{biaxial}. Furthermore,
  its form guarantees that the isotropic state $Q\equiv 0$ (or equivalently $S_1=S_2=0$) yields a global minimum at high temperatures while a uniaxial state of the form \eqref{uniaxial} where either $S_1=0,\;S_2=0$ or $S_1=S_2$ gives the minimum when temperature (i.e. the parameter $a$) is reduced below a certain critical value, cf. \cite{apala_zarnescu_01,Mottram_Newton}. In this paper we fix the temperature to be low enough so that the minimizers of $f_{LdG}$ are uniaxial. We also remark for future use that $f_{LdG}$ is bounded from below and can be made nonnegative by adding an appropriate constant. In light of this, we will henceforth assume a minimum value of zero for $f_{LdG}.$
\begin{figure}[htb]
\centering
\includegraphics[height=2in]{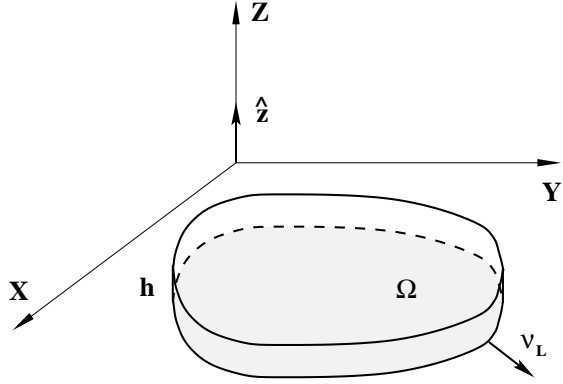}
\caption{Geometry of the problem.}
  \label{fig:1}
\end{figure}

Let $\Omega\subset\mathbb{R}^2$ be a bounded domain with a Lipschitz boundary and let $h>0$ be given (Figure \ref{fig:1}). Assume that the energy functional is
\begin{equation}
\label{energy}
E[Q]:=\int_{\Omega\times(0,h)}\left\{f_e(\nabla Q)+f_{LdG}(Q)\right\}\,dV+\int_{\Omega\times\{0,h\}}f_s(Q,\hat{z})\,dA\,,
\end{equation}
where $\hat{z}$ is a unit vector normal to the surface of the film. Given uniaxial data $g\in H^{1/2}(\partial\Omega;\mathcal{A})$ we prescribe the lateral boundary condition of the form
\begin{equation}
  \label{eq:bd}
  Q(x,y,z)=g(x,y)\ \mathrm{for}\ (x,y)\in\partial\Omega\ \mathrm{and}\ z\in (0,h).
\end{equation}

The admissible class of tensor-valued functions is then \[\mathcal{C}_h^g:=\left\{Q\in H^1\left(\Omega\times(0,h);\mathcal{A}\right):Q|_{\partial\Omega\times\{z\}}=g,\forall z\in(0,h)\right\},\] where $\mathcal{A}$ is the set of three-by-three symmetric traceless matrices defined in \eqref{eq:cala}. Throughout this work we assume that $g$ is uniaxial and is taken so that $\mathcal{C}_h^g$ is nonempty.

It has been shown, however, in \cite{ball2010nematic} that when $L_4\neq0$ minimizers of \eqref{energy} may fail to exist. On the other hand, $L_4=0$ precludes an appropriate reduction to the general form of the Oseen-Frank energy for nematics \cite{Mottram_Newton}. Since we are interested in a characterization of minimizers, we will set $L_4=0$.

We nondimensionalize the problem by scaling the spatial coordinates
\[\tilde{x}=\frac{x}{D},\ \tilde{y}=\frac{y}{D},\ \tilde{z}=\frac{z}{h},\ \]
where $D:=\mathrm{diam}(\Omega)$. Set $\xi=\frac{L_1}{2D^2},$ $\epsilon=\frac{h}{D}$ and introduce $\tilde{f}_e(Q,\nabla Q):=\frac{1}{\xi}f_e(Q,\nabla Q).$ Dropping tildes, we obtain
\begin{multline}
f_e(\nabla Q): =\left[Q_{im,j}Q_{im,j}+M_2Q_{ik,k}Q_{ij,j}+M_3Q_{ij,k}Q_{ik,j}\right]\\
+\frac{2}{\epsilon}\left[M_2Q_{ij,j}Q_{i3,3}+M_3Q_{i3,j}Q_{ij,3}\right] \\
+\frac{1}{\epsilon^2}\left[Q_{im,3}Q_{im,3}+(M_2+M_3)Q_{i3,3}Q_{i3,3}\right],
 \end{multline}
where $M_2=\frac{L_2}{L_1},$ $M_3=\frac{L_3}{L_1},$ the indices $i,m=1,2,3,$ and $j,k,l=1,2.$ Rescaling the Landau-de Gennes potential via $\tilde{f}_{LdG}(Q):=\frac{1}{w_l\xi}f_{LdG}(Q)$ and dropping tildes gives
\begin{equation}
f_{LdG}(Q)=2A\,\mathrm{tr}\left(Q^2\right)+\frac{4}{3}B\,\mathrm{tr}\left(Q^3\right)+{\left(\mathrm{tr}\left(Q^2\right)\right)}^2,
\end{equation}
where $A=\frac{a}{c},$ $B=\frac{b}{c},$ and $w_l=\frac{c}{2\xi}.$ Letting $\tilde\alpha=\frac{\alpha}{\xi},\ \tilde{\gamma}=\frac{\gamma}{\xi}$, setting
\[\tilde{f}_s(Q,\hat z):=\frac{1}{{\tilde\alpha}^2D\xi}f_s(Q,\hat z),\]
and dropping tildes, the expression for the nondimensionalized surface energy is
\begin{equation}
\label{eq:baren}
f_s(Q,\hat z)=\alpha\left[(Q\hat z\cdot\hat z)-\beta\right]^2+\gamma{\left|\left(\mathbf{I}-\hat z\otimes\hat z\right)Q\hat z\right|}^2.
\end{equation}
We conclude that the total energy is
\begin{equation}
\label{eq:8}
E[Q]=\xi D^2h\int_{\Omega\times(0,1)}\left(f_e(\nabla Q)+w_lf_{LdG}(Q)\right)\,dV+\xi D^3\int_{\Omega\times\{0,1\}}f_s(Q,\hat z)\,dA,
\end{equation}
where the rescaled domain is denoted by the same letter $\Omega$.
Introducing the non-dimensional energy $F_\epsilon[Q]:=\frac{2}{L_1h}E[Q]$, we find using \eqref{eq:8} that
\begin{equation}
\label{nden}
F_\epsilon[Q]=\int_{\Omega\times(0,1)}\left(f_e(\nabla Q)+w_lf_{LdG}(Q)\right)\,dV+\frac{1}{\epsilon}\int_{\Omega\times\{0,1\}}f_s(Q,\hat z)\,dA.
\end{equation}

\section{Convergence of minimizers of $F_\epsilon[Q]$ when $\epsilon\to 0$}
\label{s:conv}
Assume that an appropriate constant has been added to the Landau-de Gennes energy to guarantee that $F_\epsilon[Q]\geq 0$. We wish to consider a range of asymptotic regimes corresponding to different magnitudes of $\alpha$ and $\gamma$. To this end, we will assume that $\alpha=\alpha_0+\eps\alpha_1$ and $\gamma=\gamma_0+\eps\gamma_1$ for some nonnegative constants $\alpha_0, \alpha_1,\gamma_0,\gamma_1$. Then the surface energy density \eqref{eq:baren} can be written as
\begin{equation}
  \label{eq:se}
  f_s(Q,\hat z)=f_s^{(0)}(Q,\hat z)+\eps f_s^{(1)}(Q,\hat z),
\end{equation}
where 
\begin{equation}
\label{eq:fso}
f_s^{(0)}:=\alpha_0\left[(Q\hat z\cdot\hat z)-\beta\right]^2+\gamma_0{\left|\left(\mathbf{I}-\hat z\otimes\hat z\right)Q\hat z\right|}^2,
\end{equation}
and
\begin{equation}
\label{eq:fsi}
f_s^{(1)}:=\alpha_1\left[(Q\hat z\cdot\hat z)-\beta\right]^2+\gamma_1{\left|\left(\mathbf{I}-\hat z\otimes\hat z\right)Q\hat z\right|}^2.
\end{equation}
As it will become evident below, we can assume that $\alpha_0\alpha_1=\gamma_0\gamma_1=0$.
Let
 \begin{equation}
   \label{eq:f0}
   F_0[Q]:=\left\{
     \begin{array}{ll}
       \int_{\Omega}\left\{f_{e}^0(\nabla Q)+w_lf_{LdG}(Q)+2f_s^{(1)}(Q,\hat z)\right\}\,dA & \mbox{ if }Q\in H^1_g, \\
       +\infty & \mbox{ otherwise. }
     \end{array}
 \right.
 \end{equation}
Here \[f_{e}^0(\nabla Q)=Q_{im,j}Q_{im,j}+M_2Q_{ij,j}Q_{ik,k}+M_3Q_{ik,j}Q_{ij,k},\] the space \[H^1_g:=\left\{Q\in H^1(\Omega;\mathcal{D}):Q|_{\partial\Omega}=g\right\}\] and \[\mathcal{D}:=\left\{Q\in \mathcal{A}:f_s^{(0)}(Q)=0\right\},\] for some boundary data $g:\partial\Omega\to\mathcal{D}.$

We now state our main theorem on dimension reduction via $\Gamma$-convergence.  For those unfamiliar with the notion, we refer, for example, to \cite{DalMaso}. Note that whenever necessary we will view $H^1_g$ as a subset of $\mathcal{C}^g_1$ via a trivial extension to three dimensions.

\begin{theorem}
\label{t1}
Fix $g:\partial\Omega\to\mathcal{D}$ such that $H^1_g$ is nonempty. Assume that $-1< M_3<2$, and $-\frac{3}{5}-\frac{1}{10}M_3< M_2$. Then $\Gamma$-$\lim_\eps{F_\eps}=F_0$ weakly in $\mathcal{C}_1^g$. Furthermore, if a sequence $\left\{Q_\eps\right\}_{\eps>0}\subset\cg$ satisfies a uniform energy bound $F_\eps[Q_\eps]<C_0$ then there is a subsequence weakly convergent in $\cg$ to a map in $H^1_g$.
\end{theorem}

\begin{proof}First, we demonstrate that one can always choose a trivial recovery sequence. Indeed, if $Q_\eps\equiv Q\in\cg\backslash H^1_g$ then $\lim_{\eps\to0}F_\eps[Q_\eps]=+\infty=F_0[Q]$ and when $Q_\eps\equiv Q\in H^1_g$ then $F_\eps[Q_\eps]=F_0[Q_\eps]=F_0[Q]$ for all $\eps$.

For the lower semicontinuity part of the proof, consider an arbitrary sequence $\left\{Q_\eps\right\}_{\eps>0}\subset\cg$ converging weakly in $\cg$ to some $Q_0\in H^1_g.$ It has been established in \cite{Gartland_Davis} (Lemma 4.2) and \cite{Longa} that when the elastic constants satisfy the conditions $-1< M_3<2$, and $-\frac{3}{5}-\frac{1}{10}M_3< M_2$ the integral of $f_e$ is weakly lower semicontinuous in $H^1(\Omega\times(0,1))$ and, in fact,
\begin{equation}
\label{eq:coer}
f_e(\nabla Q)\geq C{|\nabla Q|}^2
\end{equation}
pointwise for all admissible $Q$, where $C>0$. Then using Sobolev embedding, one finds
\[\liminf_{\eps\to0}F_\eps[Q_\eps]\geq F_0[Q_0].\]
On the other hand, if $\left\{Q_\eps\right\}_{\eps>0}\subset\cg$ converges weakly in $\cg$ to some $Q_0\in \cg\backslash H^1_g,$ then $Q_0$ depends on $z$ and/or $Q_0$ is not $\mathcal{D}$-valued. In the first case, invoking \eqref{eq:coer}, we have
\[\liminf_{\eps\to0}F_\eps[Q_\eps]\geq C\liminf_{\eps\to0}\frac{1}{\epsilon^2}\int_{\Omega\times(0,1)}{|Q_{\eps,z}|}^2\,dV=+\infty.\]
In the second case,
\[\int_{\Omega\times\{0,1\}}f_s^{(0)}(Q_0,\hat z)\,dA\neq0,\]
(cf. \eqref{fs}), thus by strong $L^2$-convergence of traces on $\Omega\times\{0,1\}$,
\[\lim_{\eps\to0}{\frac{1}{\eps}\int_{\Omega\times\{0,1\}}f_s^{(0)}(Q_\eps,\hat z)\,dA=+\infty.}\]
Finally, since the uniform energy bound implies a uniform $H^1$-bound, there exists a subsequence weakly convergent in $H^1(\Omega\times(0,1);\mathcal{A})$ to a limit $Q_0$ that is independent of $z$. Further, strong convergence of traces in $L^2$ implies that $Q_0\in H^1_g$.
\end{proof}

\begin{remark}
  {\em When $M_2=M_3=0$, one can easily argue that the convergence of the subsequence is, in fact, strong. Indeed, $F_\eps[Q]=F_0[Q]$ for every $Q\in H^1_g$, hence $F_\eps[Q_\eps]\leq F_0[Q_0]$ for all $\eps>0$. Since
\begin{multline*}
\int_{\Omega\times(0,1)}f_{LdG}(Q_\eps)\,dV+\int_{\Omega\times\{0,1\}}f_s^{(1)}(Q_\eps,\hat z)\,dA \\ \to\int_{\Omega}\left(f_{LdG}(Q_0)+2f_s^{(1)}(Q_0,\hat z)\right)\,dA\mbox{ as }\eps\to0
\end{multline*}
we have
\[\limsup_{\eps\to0}\int_{\Omega\times(0,1)}\left({|Q_{\eps,x}|}^2+{|Q_{\eps,y}|}^2+\frac{1}{\epsilon^2}{|Q_{\eps,z}|}^2\right)\,dV\leq\int_{\Omega}\left({|Q_{0,x}|}^2+{|Q_{0,y}|}^2\right)\,dA.\]
Combining this with the lower semicontinuity of the $L^2$-norm of the derivative, strong convergence in $\cg$ follows.}
\end{remark}

\section{Minimizers of the $\Gamma$-limit in different regimes}
\label{s:reg}

In this section we explore minimization of the $\Gamma$-limit for different parameter regimes. The added penalty terms that originate from the three-dimensional surface term have a potential to disconnect uniaxial states whenever the Landau-de Gennes term becomes dominant. As we will show, this may result in formation of singular structures such as boundary layers of Allen-Cahn-type or vortices of Ginzburg-Landau-type with possible emergence of biaxiality (cf. \cite{biscari1997topological}).

In order to apply Theorem \ref{t1}, we note first that a minimizer $Q_\eps$ of $F_\eps$ exists for every $\eps>0$ by the direct method. Furthermore, selecting any $G\in H^1_g$ we observe that $F_\eps[Q_\eps]\leq F_\eps[G]=F_0[G]$ and so we can extract a weakly convergent subsequence $\{Q_\eps\}_{\eps>0}$ such that $Q_\eps\rightharpoonup Q_0$ for some $Q_0$ in $H^1_g$. Since by the basic properties of $\Gamma$-convergence any limit of minimizers is a minimizer of the limiting functional, we conclude that $Q_0$ minimizes $F_0$.

We will assume throughout this section that $\Omega$ is simply connected with a sufficiently smooth boundary. In what follows we will take $\gamma>0$ to be independent of $\eps$ (i.e., $\gamma_1=0$ in \eqref{eq:fsi}) and consider two distinct parameter regimes for $\alpha$ under several boundary conditions. 

\subsection{Regime $f_s^{(1)}\equiv0$}
\label{regime1}
First, let $\alpha_1=0$ so that $f_s^{(1)}\equiv0$ and $f_s^{(0)}$ given by \eqref{eq:fso}. Then $\mathcal{D}$-valued maps have $\hat z$ as one eigenvector with corresponding eigenvalue $\beta$. Note that these maps are not necessarily uniaxial, although they are required to be uniaxial on the boundary. 

There are two types of uniaxial $\mathcal D$-valued maps: those corresponding to $\mu=-\beta/2$ and $\mu=\beta$ in \eqref{eq:mu}, respectively. When $\mu=-\beta/2$ we have
\[Q=\frac{3\beta}{2}\left(\hat z\otimes\hat z-\frac{1}{3}\mathbf{I}\right)\]
and the uniaxial Dirichlet condition is completely rigid as $Q$ is equal to a constant. Alternatively, when $\mu=\beta$, one finds that \[Q=-3\beta\left(\mathbf{n}\otimes\mathbf{n}-\frac{1}{3}\mathbf{I}\right),\]
where $\mathbf{n}$ is an arbitrary unit vector field on the plane and one has the freedom to choose uniaxial Dirichlet boundary data of any degree. Here by the degree of $Q$ we understand the winding number of the planar $\mathbb S^1$-valued vector field
\(e^{2i\psi}\) on $\partial\Omega$ where $\mathbf n=e^{i\psi}$. Note that in this definition, it is assumed that $Q\in H^{1/2}(\partial\Omega)$ but this does not preclude the possibility of the phase $\psi$ of $\mathbf n$ jumping by an odd multiple of $\pi$ after one circulation around $\partial\Omega$. When this happens, the vector field $\mathbf n$ is discontinuous, but the field $e^{2i\psi}$ is smooth.

\medskip
\noindent{{\bf Case 1.} We begin by characterizing $Q_0$ in the first, topologically simpler case
\begin{equation}
\label{dobc}
Q|_{\partial\Omega\times(0,1)}=g:=\frac{3}{2}\beta\left(\hat z\otimes\hat z-\frac{1}{3}\mathbf{I}\right).
\end{equation}
Unless specified otherwise, we find it preferable from this point on to use the following representation of $Q\in H^1_g$ invoked, for example, in \cite{bauman_phillips_park} and motivated by simulations in \cite{PhysRevLett.59.2582}:
\begin{equation}
  \label{eq:pr}
  Q=\left(
    \begin{array}{ccc}
      p_1-\frac{\beta}{2} & p_2 & 0 \\
      p_2 & -p_1-\frac{\beta}{2} & 0 \\
      0 & 0 & \beta
    \end{array}
\right).
\end{equation}
It is a convenient change of variables in the setting when one eigenvector of the $Q$-tensor is parallel to the $z$-axis. Note that the boundary condition in the new coordinates is
\begin{equation}
  \label{eq:bcpq}
  \mathbf{p}|_{\partial\Omega}=(p_1,p_2)|_{\partial\Omega}=\mathbf{0}.
\end{equation}
Applying the identity
\begin{multline}
\label{eq:bpp}
f_e^0(Q)=\frac{1}{2}\left(2+M_2+M_3\right)|\nabla {\mathbf p}|^2+\frac{1}{8}\left(6+M_2+M_3\right)|\nabla\beta|^2 \\
+\frac{M_2+M_3}{2}\left(p_{1x}\beta_x-p_{1y}\beta_y+p_{2y}\beta_x+p_{2x}\beta_y\right)+\left|M_2+M_3\right|\left(p_{1x}p_{2y}-p_{1y}p_{2x}\right)
\end{multline}
from \cite{bauman_phillips_park}, the expression for $F_0[Q]$ takes the form
\begin{equation}
  \label{eq:f0pq}
  \frac{1}{M}F_0[Q]=\tilde F_0[\mathbf{p}]:=\int_{\Omega}\left\{\frac{1}{2}{\left|\nabla\mathbf{p}\right|}^2+\frac{1}{\delta^2}W(|\mathbf{p}|)\right\}\,dV,
\end{equation}
where we have used that $\beta$ is constant and that the integral of the Jacobian of $\mathbf p$ vanishes due to \eqref{eq:bcpq}. Here $\mathbf{p}=\left(p_1,p_2\right),$ the parameters $M=2+M_2+M_3$, $\delta=\sqrt{M/w_l}$, and
\begin{equation}
  \label{eq:tc}
  W(t)=4t^4+\tilde{C}t^2+\tilde{D}
\end{equation}
with $\tilde{C}=6\beta^2-4B\beta+4A$ and $\tilde{D}\in\mathbb{R}$. Note that $\tilde C$, in particular, varies with temperature through its dependence on the coefficient $A$. It is plausible to assume that $\tilde C$ may change its sign in appropriate circumstances.

One easily observes that the minimizer of \eqref{eq:f0pq}, subject to the boundary condition \eqref{eq:bcpq} has a constant phase and so \eqref{eq:f0pq} reduces to a scalar minimization problem for the modulus $p:=|\mathbf{p}|$. The minimizers of  \eqref{eq:f0pq} satisfy the Allen-Cahn type equation
\begin{equation}
  \label{eq:whatever}
  -\Delta p+\frac{1}{\delta^2}W^\prime(p)=0\mbox{ in }\Omega,\quad p=0\mbox{ on }\partial\Omega.
\end{equation}
The function $p\equiv0$ always solves this problem and, in fact, is the unique critical point and thus the minimizer if $\tilde{C}\geq0$. Therefore the minimizing $Q$-tensor in this regime corresponds to a constant uniaxial state
\begin{equation}
  \label{eq:qbeta}
  Q_\beta=\left(
    \begin{array}{ccc}
      -\frac{\beta}{2} & 0 & 0 \\
      0 & -\frac{\beta}{2} & 0 \\
      0 & 0 & \beta
    \end{array}
\right).
\end{equation}
(cf.  \eqref{eq:pr}) matching the boundary condition \eqref{dobc}. When $\tilde{C}<0$, however, the constant state $Q\equiv Q_\beta$ looses stability once $\lambda_1(\Omega)$ exceeds $-2\tilde{C}/\delta^2$, where $\lambda_1(\Omega)$ is the first eigenvalue of the Laplacian. A minimizing nonconstant solution emerges in this parameter regime when the value of $p$ enforced on $\partial\Omega$ by the surface energy does not minimize the bulk Landau-de Gennes energy. Indeed we expect a boundary layer to form in the vicinity of $\partial\Omega$, bridging $p=0$ to the minimum value of $W$ in the bulk when $\delta\ll1$ see e.g. \cite{Berger_70}.

Further, the nonconstant minimizing configuration cannot be uniaxial everywhere in $\Omega$ as can be seen from the measure of biaxiality introduced in \cite{kaiser1992stability}:
\begin{equation}
  \label{eq:bia}
  \xi(Q)^2:=1-6\frac{{\left(\tr{Q}^3\right)}^2}{{\left(\tr{Q}^2\right)}^3},
\end{equation}
where $\xi(Q)=0$ implies that $Q$ is uniaxial. If we express $\xi$ in terms of $p$ and $\beta$, we have
\begin{equation}
  \label{eq:biap}
  \xi(p,\beta)=1-27\,\frac{\beta^2\left(4p^2-\beta^2\right)^2}{(4p^2+3\beta^2)^3}.
\end{equation}
Since $\beta$ is fixed, there are finitely many values of $p$ where $\xi(p,\beta)$ vanishes. We conclude that if $p$ is a nontrivial solution of the elliptic boundary value problem \eqref{eq:whatever} then the minimizer is necessarily biaxial almost everywhere.

\medskip
\noindent{{\bf Case 2.} Now we turn to the case of the boundary condition
\begin{equation}
\label{nbc}
Q|_{\partial\Omega\times(0,1)}=g:=-3\beta\left(\bn\otimes\bn -\frac{1}{3}\mathbf{I}\right),
\end{equation}
where $\bn:\partial\Omega\to\mathbb{S}^1$ is such that $\bn\otimes\bn$ is smooth and may have a nonzero degree in the sense described above. Then the tensor $Q_0$ minimizes the energy $F_0$ given by \eqref{eq:f0pq}, subject to the boundary condition
\begin{equation}
\label{nbc_0}
Q|_{\partial\Omega}=-3\beta\left(\bn\otimes\bn -\frac{1}{3}\mathbf{I}\right).
\end{equation}
Using the representation \eqref{eq:pr}, we have that $Q_0$ can be represented by the vector $\mathbf{p}_0$ that minimizes $\tilde{F}_0$ and satisfies
\begin{equation}
  \label{eq:pbc}
  \mathbf{p}=-3\beta\left(n_1^2-\frac{1}{2},n_1n_2\right)
\end{equation}
on $\partial\Omega$ where $|\mathbf{p}|=\frac{3\beta}{2}$. In fact, $\tilde{F}_0$ is altered by an additive constant depending on $g$ due to the presence of the null Lagrangian $p_{1x}p_{2y}-p_{1y}p_{2x}$ in \eqref{eq:bpp} that will not affect minimization.

In order to better understand the behavior of $Q_0$ let $\delta$ be small. Then, in general, when $\tilde{C}$ in \eqref{eq:tc} is negative, the corresponding variational problem is of Ginzburg-Landau-type, but with a boundary layer bridging the equilibrium value $\sqrt{-\tilde C/8}$ of $p$ in the bulk to that of $3\beta/2$ enforced by the surface energy on $\partial\Omega$. Furthermore, for topologically nontrivial boundary data for $\mathbf{p}$, the minimizing vector field $\mathbf{p}_0$ has to vanish somewhere within a vortex core structure of a characteristic size $\delta$ in $\Omega$. Recalling \eqref{eq:pr}, we have that $\mathbf{p}_0(x)=\mathbf{0}$ corresponds to a uniaxial state at $x$ with the director pointing along $\hat z$, namely
\[Q_0(x)=\frac{3}{2}\beta\left(\hat z\otimes\hat z-\frac{1}{3}\mathbf{I}\right).\]

In the case of $\tilde{C}\geq0$---making $W$ convex with a unique minimum at $p=0$---we expect again that a boundary layer would form along $\partial\Omega$ bridging the boundary value of $p$, namely $3\beta/2$, to zero in the interior of $\Omega$. In all cases discussed in this section, we expect that symmetry breaking caused by surface energy will induce biaxiality within small sets where $Q_0$ experiences large variations.

\begin{remark}
It is also possible to consider an intermediate asymptotic regime with non-zero degree Dirichlet data where minimality does not force biaxiality. Suppose the surface energy is taken to be substantially smaller than in those cases considered so far but where the relative strength of the Landau-de Gennes contribution is stronger. Suppose, for example,
that $\gamma=\eps\gamma_1$ and $\alpha_0=\alpha_1=0$ for some $\gamma_1>0$ in expressions \eqref{eq:fso}, \eqref{eq:fsi} while say $w_l=\frac{w_0}{\eps}$ for some $w_0>0.$ In view of \eqref{nden} one
sees that the new $\Gamma$-limit would now take the form
 \begin{equation}
   \label{eq:f0r}
   F_0[Q]:=\left\{
     \begin{array}{ll}
       \int_{\Omega}\left\{f_{e}^0(\nabla Q)+2\gamma_0{\left|\left(\mathbf{I}-\hat z\otimes\hat z\right)Q\hat z\right|}^2\right\}\,dA & \mbox{ if }Q\in \tilde{H}^1_g, \\
       +\infty & \mbox{ otherwise. }
     \end{array}
 \right.
 \end{equation}
where $\tilde{H}^1_g:=\{Q\in H^1(\Omega;\mathcal{A}):\;Q|_{\partial\Omega}=g,\;f_{LdG}(Q)=0\}$.
Here we assume the uniaxial data $g$ takes the form
\[
g=S^*\left(\bn\otimes\bn -\frac{1}{3}\mathbf{I}\right)
\] for some planar vector field $\bn$ where $S^*$
corresponds to the preferred value dictated by the requirement $f_{LdG}(g)=0$ and suppose $g$ has even degree so as to ensure that $\tilde{H}^1_g$ is nonempty. In this case,
say for $\gamma_0$ large, one expects the uniaxial minimizer to simply undergo an out of plane rotation of its director in a neighborhood of the boundary, keeping the degree of orientation fixed, thereby
smoothly bridging the boundary value ${\bf n}$ to $\hat{z}$ inside the domain so as to accommodate the cost of the remnant of the surface energy term residing in the $\Gamma$-limit.
\end{remark}

\subsection{Regime $f_s^{(1)}(Q,\hat z)=\alpha_1\left[(Q\hat z\cdot\hat z)-\beta\right]^2$}
Now consider the case when $\alpha=\eps \alpha_1$ and $\gamma=\gamma_0$ for some $\gamma_0,\alpha_1>0$. Then $f_s^{(0)}(Q,\hat z)=\gamma_0{\left|\left(\mathbf{I}-\hat z\otimes\hat z\right)Q\hat z\right|}^2$ and $f_s^{(1)}(Q,\hat z)=\alpha_1\left[(Q\hat z\cdot\hat z)-\beta\right]^2$ and the set $\mathcal{D}$ consists of traceless symmetric tensors having $\hat z$ as one of its eigenvectors.  When $\alpha_1=0$ the asymptotic behavior of the limiting functional $F_0$ was considered in \cite{bauman_phillips_park} as $\delta\to0$. They characterize the minimizer by identifying finitely many interacting vortices whose positions are determined via minimization of a renormalized energy of the type introduced in \cite{BBH} for the Ginzburg-Landau functional. This description is very similar to the already discussed regime above which formally corresponds to $\alpha_1=\infty$, i.e., an eigenvalue corresponding to $\hat z$ is identically equal to $\beta$ in $\Omega$. Hence we expect similar behavior for any finite value of $\alpha_1$.

\section{Structure of the singular region for topologically nontrivial boundary data}
In this section we provide further insight into the structure of the singular region that develops for topologically nontrivial boundary data in Case 2 of Section \ref{regime1} when $\tilde C<0$. 

Suppose for the moment that the dimensional reduction was carried out as in Theorem \ref{t1} for a sequence of functionals without the surface energy term. In this case, the conclusion of the theorem would remain the same, with the exception that no additional constraints would have to be imposed on the target space of admissible tensors for the limiting problem. In particular, any tensor field described by a field of symmetric traceless $3\times3$ matrices in $H^1(\Omega)$ satisfying the appropriate boundary conditions on $\partial\Omega$ would be admissible. 

Now recall that any tensor representing a uniaxial nematic as in \eqref{uniaxial} can be associated with a point on $\mathbb{RP}^2$, i.e., on $\mathbb{S}^2$ with antipodal points identified (since ${\bf n}$ and $-{\bf n}$ lead to the same $Q$-tensor). This means that the tensor field
 $Q|_{\partial\Omega}$ corresponding to the $\mathbb{S}^1$-valued boundary data \eqref{nbc_0} can be associated with one or more semi-circular arcs along the equator of $\mathbb{S}^2$. If the degree of ${\bf p}$ in \eqref{eq:pbc} is an even integer, then in terms of the image in $\mathbb{RP}^2$ this corresponds to an even number of such half-equators. One can then smoothly deform the $Q$-tensor within $\Omega$ through a field of uniaxial nematics as indicated in Figure \ref{fig:even} to, say, a point in the interior of $\Omega$ where the director points either north or south. That is, the even number of semi-circles originally overlapping on the equator can gradually migrate towards the poles as a family of closed curves on $\mathbb{RP}^2$. 
 
Note that a sufficiently smooth uniaxial nematic tensor field can be described by a smooth director field \cite{Ball_Zarn} with the boundary data of degree equal to the half of that for $Q|_{\partial\Omega}$. Hence the uniaxial nematic tensor field can be constructed even if the degree of the corresponding director field on the boundary is non-zero. Putting this another way, the director is $\mathbb{S}^2$-valued in $\Omega$ and its $\mathbb{S}^1$-valued topologically nontrivial restriction to $\partial\Omega$ can be bridged without forming a singularity in the interior of $\Omega$ by the director ``escaping into the third dimension".
\begin{figure}[htb]
\centering
\includegraphics[height=2.5in]{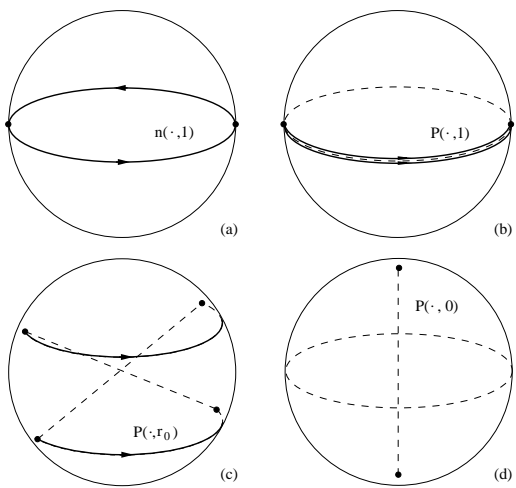}
\caption{A uniaxial nematic tensor $Q$ with the director ${\bf n}$ can be identified with a projection matrix $\mathrm{P}={\bf n}\otimes{\bf n}$, up to a translation and dilation. Let $\Omega=B_1(0)\subset\mathbb{R}^2$ and let $s$ be some parametrization of a circle of radius $r\in(0,1]$. (a) The boundary data ${\bf n}|_{\partial\Omega}=e^{is}$ corresponds to an equator in $\mathbb{S}^2$; (b) The boundary data for $\mathrm{P}$ corresponds to two half-equators in $\mathbb{S}^2$; (c) For $0<r_0<1$, the half-circles migrate toward the north and south poles of $\mathbb{S}^2$, respectively; (d) At $r=0$ half-circles contract to the poles. This corresponds to the director pointing up or down at $r=0$.}
  \label{fig:even}
\end{figure}

If the degree of ${\bf p}$ in \eqref{eq:pbc} is an odd integer, then the image of $\partial\Omega$ in $\mathbb{RP}^2$ corresponds to an odd number of half-equators of $\mathbb{S}^2$. This curve is not contractible in $\mathbb{RP}^2$, i.e., it cannot be smoothly deformed into a pair of points in $\mathbb{S}^2$. As Figure \ref{fig:odd} illustrates, the shortest smooth deformation of an odd number of half-equators of $\mathbb{S}^2$ is a single half-equator in $\mathbb{S}^2$, that is a closed geodesic in $\mathbb{RP}^2$. Then a uniaxial nematic $Q$-tensor field with non-contractible boundary data on $\partial\Omega$ cannot be smooth everywhere in $\Omega$ and must have a singularity at some $x\in\Omega$ as $Q$ will have to assume all values on a single half-equator in $\mathbb{S}^2$ in an arbitrarily small neighborhood of $x$.
\begin{figure}[htb]
\centering
\includegraphics[height=2.5in]{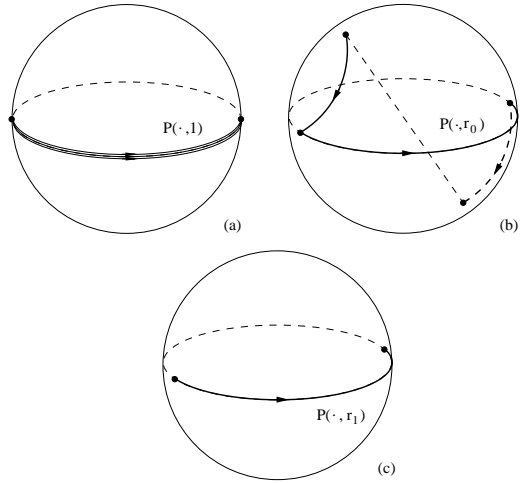}
\caption{Behavior of $Q$ for the boundary data of degree $3$. The setup is the same as in the caption to Figure \ref{fig:even} and $0<r_1\ll1$.}
  \label{fig:odd}
\end{figure}

Staying with this broader class of admissible tensors for the moment, if we now let $w_l$ in \eqref{eq:f0} be large but finite (this is equivalent to setting $\delta$ in \eqref{eq:f0pq} small), any deviation of $Q$ from a uniaxial nematic state--a minimum of $f_{LdG}$-- will result in large energy penalty. The smooth energy-minimizing tensor field will then be approximately uniaxial nematic everywhere. That is, as described above, if the degree of $Q|_{\partial\Omega}$ is even then there is no topological obstruction preventing $Q_0$ from being almost uniaxial nematic everywhere in $\Omega$.  Since $Q_0$ has a fixed set of eigenvalues that minimize $f_{LdG}$ and the director vector is an eigenvector of $Q$ corresponding to the eigenvalue with the largest magnitude, the variations of $Q$ throughout the domain $\Omega$ can be interpreted as rotations of the eigenframe of $Q$. On the other hand, if the degree of $Q|_{\partial\Omega}$, i.e. the degree of ${\bf p}$ in \eqref{eq:pbc}, is an odd integer, then no such smooth uniaxial deformation exists going into $\Omega$ and there will be a small region of the size $1/\sqrt{w_l}$ where the topological constraint will force $Q_0$ to be isotropic and/or biaxial (cf. \cite{GM,canevari13}).

Let us now contrast this more relaxed target space with what occurs in our investigation where the full energy functional contains the surface energy term. The effect of this is that the admissible tensor fields $Q$ for the limiting problem are now $\mathcal D$-valued.
\begin{figure}[htb]
\centering
\includegraphics[height=1.5in]{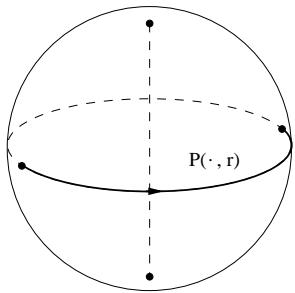}
\caption{Geometry of the target space of uniaxial nematic $\mathcal D$-valued maps. Only the values along the equator and at the poles of $\mathbb{S}^2$ are permitted.}
  \label{fig:dclass}
\end{figure}
Since the normal to the surface of the film is then one of the eigenvectors of $Q$, there are only two types of possible uniaxial nematic states for admissible tensors---those with the director either perpendicular or parallel to the surface of the film (Figure \ref{fig:dclass}). If $\delta$ in \eqref{eq:f0pq} is small, the potential term $W$ would force the liquid crystal to be uniaxial nematic throughout most of $\Omega$, thus making the director either perpendicular or parallel to the surface normal $\hat z$. This orientational constraint is due to the strong influence of the surface energy and it makes director escape through the sequence of nematic states extremely costly. Given any topologically nontrivial planar boundary data for the director on $\partial\Omega$, we expect the director to remain planar as much as possible thus leading to formation of a localized region or regions with large gradients inside $\Omega$ for {\it any} choice of topologically nontrivial ${\bf p}_{\partial\Omega}$.

At the core of any singular region that may develop, one would expect that the orientations of the molecules would be parallel to the surface of the film, while remaining random within that surface. The orientational distribution function then is axially symmetric with respect to the normal to the film and the corresponding $Q$-tensor is uniaxial with the director being parallel to the normal (cf. \cite{virga_talk}). Since only two types of uniaxial states are permitted for admissible tensors, the connection between the uniaxial states in the core and away from the core has to occur through a sequence of biaxial states via the mechanism of so-called eigenvalue exchange \cite{palffy1994new}.

\section{Conclusions}
In this paper we have considered a rigorous dimension reduction procedure for nematic films described within the Landau-de Gennes model. We have established a general $\Gamma$-convergence result for the sequence of non-dimensional Landau-de Gennes energies as a small parameter $\eps$ related to the film thickness tends to zero. Although this result is applicable to any combination of coefficients in the expression \eqref{bare} for the surface energy, we only consider cases when a $Q$-tensor minimizing $f_s(Q,\nu)$ in \eqref{bare} has the normal to the surface of the film as one of its eigenvectors.

In the thin-film limit, the dominant contributions to the energy come from the bulk elastic energy terms containing the derivatives in the direction normal to the surface of the film as well as from the surface energy terms that are independent of $\eps$. The energy penalty due to normal derivative terms forces the limiting energy-minimizing $Q$-tensor field to be independent of the spatial variable perpendicular to the surface of the film and reduces the domain of the problem from three to two dimensions. The presence of the surface energy terms that are independent of $\eps$ imposes constraints on the target space of admissible three-by-three $Q$-tensors. The resulting limiting bulk energy defined over the constrained set of $Q$-tensors consists of lower order bulk and surface energy contributions.

Our results, along with those in \cite{bauman_phillips_park}, demonstrate that the limiting problems for various parameter regimes can be studied by using the techniques developed for scalar Allen-Cahn type problems or Ginzburg-Landau-type vector-valued problems. Depending on the relationship between the coefficients of the surface and bulk energy terms, minimizing $Q$-tensor fields can develop boundary layers. For topologically nontrivial boundary data and large $w_l$ in \eqref{eq:f0}, the $Q$-tensor fields are characterized by vortices with characteristic core size $\sim1/\sqrt{w_l}.$ Note that even though $1/\sqrt{w_l}$ in this case is small, it is still much larger than $\eps$ and two-dimensional vortices correspond to disclination lines perpendicular to the surface of the film.

\section{Acknowledgements}
D.G. acknowledges support from NSF DMS-1434969. P.S. acknowledges support from NSF DMS-1101290 and NSF DMS-1362879.

\bibliographystyle{ieeetr}
\bibliography{tensor_nematic}

\end{document}